\documentclass[final,reqno]{siamltex}
\usepackage{latexsym,amsmath,amssymb,amsfonts,mathrsfs}
\usepackage{epsf,graphicx,epsfig,color,cite,cases}
\usepackage{subfigure,graphics,multirow,marginnote,enumerate,bm}
\newcommand{\R}{{\mat R}}

\newcommand{\C}{{\mat C}}
\newcommand{\E}{{\mat E}}

\newcommand{\ds}{\displaystyle}
\newcommand{\no}{\nonumber}
\newcommand{\be}{\begin{eqnarray}}
\newcommand{\ben}{\begin{eqnarray*}}
\newcommand{\en}{\end{eqnarray}}
\newcommand{\enn}{\end{eqnarray*}}

\newcommand{\pa}{\partial}

\newcommand{\ov}{\overline}
\newcommand{\I}{{\rm Im}}
\newcommand{\Rt}{{\rm Re}}
\newcommand{\curl}{{\rm curl}}
\newcommand{\Curl}{{\rm Curl}}

\newcommand{\Dive}{{\rm Div}}
\newcommand{\g}{\gamma}
\newcommand{\G}{\Gamma}

\newcommand{\vep}{\varepsilon}
\newcommand{\Om}{\Omega}
\newcommand{\om}{\omega}
\newcommand{\sig}{\sigma}

\newcommand{\al}{\alpha}
\newcommand{\bt}{\beta}

\newcommand{\ti}{\times}

\newcommand{\wid}{\widetilde}

\newcommand{\na}{\nabla}
\newcommand{\mat}{\mathbb}
\newcommand{\se}{\setminus}
\newcommand{\ify}{\infty}

\newcommand{\la}{\lambda}

\newcommand{\les}{\lesssim}
\newcommand{\ch}{\check}
\newcommand{\V}{\Vert}

\newcommand{\0}{\bm{0}}

\newcommand{\on}{\text{on}}
\newcommand{\gin}{\text{in}}
\newcommand{\PML}{{\rm PML}}

\newtheorem{remark}[theorem]{Remark}

\begin{document}
\renewcommand{\theequation}{\arabic{section}.\arabic{equation}}

\title{\bf
Convergence analysis of the PML method for time-domain electromagnetic scattering problems}
\author{Changkun Wei\thanks{Academy of Mathematics and Systems Science, Chinese Academy of Sciences,
Beijing 100190, China  and School of Mathematical Sciences, University of Chinese
Academy of Sciences, Beijing 100049, China ({\tt weichangkun@amss.ac.cn})}
\and
Jiaqing Yang\thanks{School of Mathematics and Statistics, Xi'an Jiaotong University,
Xi'an, Shaanxi, 710049, China ({\tt jiaq.yang@mail.xjtu.edu.cn; jiaqingyang@amss.ac.cn})}
\and
Bo Zhang\thanks{NCMIS, LSEC, and Academy of Mathematics and Systems Sciences, Chinese Academy of Sciences,
Beijing, 100190, China and School of Mathematical Sciences, University of Chinese Academy of Sciences,
Beijing 100049, China ({\tt b.zhang@amt.ac.cn})}
}
\date{}
\maketitle


\begin{abstract}
In this paper, a perfectly matched layer (PML) method is proposed to solve the
time-domain electromagnetic scattering problems in 3D effectively.
The PML problem is defined in a spherical layer
and derived by using the Laplace transform and real coordinate stretching in the frequency domain.
The well-posedness and the stability estimate of the PML problem are first proved based on the Laplace
transform and the energy method. The exponential convergence of the PML method is then established
in terms of the thickness of the layer and the PML absorbing parameter.
As far as we know, this is the first convergence result for the time-domain PML method for
the three-dimensional Maxwell equations.
Our proof is mainly based on the stability estimates of solutions of the truncated PML problem
and the exponential decay estimates of the stretched dyadic Green's function for the Maxwell
equations in the free space.
\end{abstract}



\pagestyle{myheadings}
\thispagestyle{plain}
\markboth{C. Wei, J. Yang and B. Zhang}{Convergence of the PML method for time-domain electromagnetic scattering}

\section{Introduction}
\setcounter{equation}{0}

In this paper, we consider time-domain electromagnetic scattering problems by a perfectly conducting
obstacle, of which the well-posedness and stability of solutions have been established in \cite{chen2008}.
The purpose of this paper is to propose a perfectly matched layer (PML) method for
solving the time-domain electromagnetic scattering problem effectively.

Recently, time-dependent scattering problems have attracted much attention due to their
capability of capturing wide-band signals and modeling more general materials and
nonlinearity \cite{Li2012,Chen2014,Wang2012}.
For example, the well-posedness and stability analysis can be found
in \cite{chen2008,LLA2015,GL2016,GL2017} for time-domain electromagnetic scattering problems
by bounded obstacles, diffraction gratings and unbounded structures,
and in \cite{BGL2018,Hsiao2015,GLZ2017,wy2019} for acoustic-elastic interaction problems,
including the case of bounded elastic bodies in a locally perturbed half-space
and the case of unbounded layered structures.

The PML method was originally proposed by B\'erenger in 1994 for solving the time-dependent Maxwell's
equations \cite{Berenger1994}. The purpose of the PML method is to surround the computational domain
with a specially designed medium in a finite thickness layer in which the scattered waves decay
rapidly regardless of the wave incident angle, thereby greatly reducing the computational complexity
of the scattering problem. Since then, a large amount of work have been done on the construction
of various structures of PML absorbing layers for solving scattering problems
(see, e.g., \cite{Bramble2007,TURKEL1998,Collino1998,DeHoop2001,TC2001,DJ2006}).
On the other hand, convergence analysis of the PML method has also been studied by
many authors for time-harmonic scattering problems. For example, the exponential convergence has been established in terms of the thickness of the PML layer in \cite{Lassas1998,HSZ2003,CZ2010,BP2013} for
time-harmonic acoustic scattering problems, and in \cite{BW2005,Bramble2007,BP2008,BP2012,CZ2017,LWZ2011}
for time-harmonic electromagnetic scattering problems
including the two-layer medium case \cite{CZ2017} and the case with unbounded surfaces \cite{LWZ2011}.
There are also some work on the adaptive PML finite element method which provides
a complete numerical strategy for solving unbounded scattering problems within the framework
of the finite element method \cite{CW2003,CC2007,CL2005,CW2005}.

Compared with the time-harmonic PML method, very few theoretical results are available for
the analysis of the time-domain PML method for time-domain scattering problems.
For the time-domain acoustic scattering problems in 2D, the exponential convergence of a circular
PML method was proved in \cite{chen2009} in terms of the thickness and absorbing parameters of
the PML layer, based on the exponential decay of the modified Bessel function.
An uniaxial PML method was proposed in \cite{CW2012} for time-domain acoustic scattering problems
in two dimensions, based on the Laplace transform and complex coordinate stretching in the frequency
domain, and and its exponential convergence was also established in terms of the thickness and
absorbing parameters of the PML layer. In addition, the well-posedness and stability estimates
of the time-domain PML method have been proved in \cite{BGL2018} for the two-dimensional
acoustic-elastic interaction problems.
To the best of our knowledge, no theoretical analysis result is available so far
for the time-domain PML method for the three-dimensional electromagnetic scattering problems.

The purpose of this paper is to provide a theoretical study of the time-domain PML method
for the three-dimensional electromagnetic scattering problems, including its well-posedness and
stability as well as its exponential convergence in terms of the thickness and absorbing parameters
of the PML layer. Different from the complex coordinate stretching technique based on the Laplace
transform variable $s^{-1}$ in \cite{chen2009,CW2012}, we construct the PML layer by using
a real coordinate stretching technique associated with $[\Rt(s)]^{-1}$ in the frequency domain.
The existence, uniqueness and stability estimates of the PML problem are first established,
based on the Laplace transform and the energy method.
By analyzing the exponential decay properties of the stretched dyadic Green's function in the PML layer
in conjunction with the well-posedness of solutions of the truncated PML problem,
the exponential convergence of the PML method is then proved in terms of the thickness and absorbing
parameters of the PML layer.

The remaining part of this paper is as follows. In Section \ref{sec2}, we introduce some basic tools
including the Laplace transform and some Sobolev spaces needed in this paper.
The time-domain electromagnetic scattering problem is presented in Section \ref{sec3},
including the well-posedness of the problem and some properties of the transparent boundary
condition (TBC) established in \cite{chen2008}. Section \ref{sec4} is devoted to the well-posedness
and stability estimates of the truncated PML problem. The exponential convergence of the PML method
is established in Section \ref{sec5}, while some conclusions are given in Section \ref{sec6}.

\section{The Laplace transform and Sobolev spaces}\label{sec2}
\setcounter{equation}{0}

In this section we introduce the Laplace transform and the Sobolev spaces needed in this paper.

\subsection{The Laplace transform}

For each $s=s_1+is_2\in\C_+$ with $s_1>0,\;s_2\in\R$, the Laplace transform of the vector field
$\bm{u}(t)$ is defined as
\ben
\ch{\bm{u}}(s)=\mathscr{L}(\bm{u})(s)=\int_0^{\infty}e^{-st}\bm{u}(t)dt.
\enn
It is easy to verify that the Laplace transform has the following properties:
\be\label{A.1}
\mathscr{L}(\bm{u}_t)(s)&=&s\mathscr{L}(\bm{u})(s)-\bm{u}(0),\\ \label{A.2}
\int_0^t\bm{u}(\tau)d\tau&=&\mathscr{L}^{-1}(s^{-1}\ch{\bm{u}})(t),
\en
where $\mathscr{L}^{-1}$ denotes the inverse Laplace transform.

Now, by the definition of the Fourier transform we have that for any $s_1>0$,
\ben
\mathscr{F}(\bm{u}(\cdot)e^{-s_1\cdot})(s_2)&=&\int_{-\infty}^{+\infty}\bm{u}(t)e^{-s_1t}e^{-is_2t}dt
=\int_{0}^{\infty}\bm{u}(t)e^{-(s_1+is_2)t}dt\\
&=&\mathscr{L}(\bm{u})(s_1+is_2),\;\;\;s_2\in\R.
\enn
Then it follows by the formula of the inverse Fourier transform that for any $s_1>0$,
\ben
\bm{u}(t)e^{-s_1t}=\mathscr{F}^{-1}\{\mathscr{F}(\bm{u}(\cdot)e^{-s_1\cdot})\}
=\mathscr{F}^{-1}\Big(\mathscr{L}(\bm{u}(s_1+is_2))\Big),
\enn
that is,
\be\label{A.4}
\bm{u}(t)=\mathscr{F}^{-1}\Big(e^{s_1t}\mathscr{L}(\bm{u}(s_1+is_2))\Big),\;\;s_1>0,
\en
where $\mathscr{F}^{-1}$ denotes the inverse Fourier transform with respect to $s_2$.

By (\ref{A.4}), the Plancherel or Parseval identity for the Laplace transform can be obtained
(see \cite[(2.46)]{Cohen2007}).

\begin{lemma}{\rm (The Parseval identity).}
If $\ch{\bm{u}}=\mathscr{L}(\bm{u})$ and $\ch{\bm{v}}=\mathscr{L}(\bm{v})$, then
\be\label{A.5}
\frac{1}{2\pi}\int_{-\ify}^{\ify}\ch{\bm{u}}(s)\cdot\ch{\bm{v}}(s)ds_2
=\int_0^{\ify}e^{-2s_1t}\bm{u}(t)\cdot\bm{v}(t)dt
\en
for all $s_1>\la$, where $\la$ is the abscissa of convergence for the Laplace transform
of $\bm{u}$ and $\bm{v}$.
\end{lemma}

The following lemma was proved in \cite{treves1975} (see \cite[Theorem 43.1]{treves1975}).

\begin{lemma}{\rm \cite[Theorem 43.1]{treves1975}.}
Let $\ch{\bm{\om}}(s)$ denote a holomorphic function in the half complex plane $s_1=\Rt(s)>\sig_0$
for some $\sig_0\in\R$, valued in the Banach space $\E$. Then the following statements are equivalent:
\begin{enumerate}[1)]
\item there is a distribution $\om\in\mathcal{D}_+^{'}(\E)$ whose Laplace transform is equal
to $\ch{\bm{\om}}(s)$, where $\mathcal{D}_+^{'}(\E)$ is the space of distributions on the real line which vanish identically
in the open negative half-line;
\item there is a $\sig_1$ with $\sig_0\leq\sig_1<\ify$ and an integer $m\geq0$ such that for
all complex numbers $s$ with $s_1=\Rt(s)>\sig_1$ it holds that $\V\ch{\bm{\om}}(s)\V_{\E}\les(1+|s|)^m$.
\end{enumerate}
\end{lemma}

\subsection{Sobolev spaces}

For a bounded domain $D\subset\R^3$ with Lispchitz continuous boundary $\Sigma$, the Sobolev space
$H(\curl,D)$ is defined by
\ben
H(\curl,D):=\{\bm{u}\in L^2(D)^3:\;\nabla\times\bm{u}\in L^2(D)^3\}
\enn
which is a Hilbert space equipped with the norm
\ben
\|\bm{u}\|_{H(\curl,D)}=\left(\|\bm{u}\|^2_{L^2(D)^3}+\|\nabla\times\bm{u}\|^2_{L^2(D)^3}\right)^{1/2}.
\enn

Denote by $\bm{u}_{\Sigma}=\bm{n}\times(\bm u\times\bm{n})|_\Sigma$ the tangential component of $\bm u$
on $\Sigma$, where $\bm n$ denotes the unit outward normal vector on $\Sigma$.
By \cite{BCS2002} we have the following bounded surjective trace operators:
\ben
&&\g:\;H^1(D)\rightarrow H^{1/2}(\Sigma),\quad\g\varphi=\varphi\quad\on\;\;\Sigma,\\
&&\g_t:\;H(\curl,D)\rightarrow H^{-1/2}(\Dive,\Sigma),\quad\g_t\bm u=\bm u\times\bm{n}\quad\on\;\;\Sigma,\\
&&\g_T:\;H(\curl,D)\rightarrow H^{-1/2}(\Curl,\Sigma),\quad
\g_T\bm u=\bm{n}\times(\bm u\times\bm{n})\quad\on\;\;\Sigma,
\enn
where $\g_t$ and $\g_T$ are known as the tangential trace and tangential components trace operators,
and $\Dive$ and $\Curl$ denote the surface divergence and surface scalar curl operators,
respectively (for the detailed definition of $H^{-1/2}(\Dive,\Sigma)$ and $H^{-1/2}(\Curl,\Sigma)$,
we refer to \cite{BCS2002}). By \cite{BCS2002} again we know that $H^{-1/2}(\Dive,\Sigma)$ and
$H^{-1/2}(\Curl,\Sigma)$ form a dual pairing satisfying the integration by parts formula
\be\label{curl_form}
(\bm{u},\nabla\times\bm{v})_{D}-(\nabla\times\bm{u},\bm{v})_{D}
=\left\langle\g_{t}\bm{u},\g_{T}\bm{v}\right\rangle_{\Sigma}\quad\forall\;\bm{u},\bm{v}\in\bm{H}(\curl,D),
\en
where $(\cdot,\cdot)_{D}$ and $\langle\cdot,\cdot\rangle_{\Sigma}$ denote the $L^2$-inner product
on $D$ and the dual product between $H^{-1/2}(\Dive,\Sigma)$ and $H^{-1/2}(\Curl,\Sigma)$, respectively.

For any $S\subset\Sigma$, the subspace with zero tangential trace on $S$ is denoted as
\ben
H_S(\curl,D):=\left\{\bm{u}\in H(\curl,D):\g_{t}\bm{u}=0\;\;\text{on}\;\;S\right\}.
\enn
In particular, if $S=\Sigma$ then we write $H_{0}(\curl,D):=H_{\Sigma}(\curl,D)$.

\section{The scattering problem}\label{sec3}
\setcounter{equation}{0}

We consider the time-domain electromagnetic scattering problem with the perfectly conducting
boundary condition on the boundary of the obstacle:
\be\label{2.1}
\begin{cases}
\ds\nabla\times\bm{E}+\mu\frac{\pa\bm{H}}{\pa t}=\0&\gin\;\;\;(\R^3\se\ov{\Om})\times(0,T),\\
\ds\nabla\times\bm{H}-\vep\frac{\pa\bm{E}}{\pa t}=\bm J&\gin\;\;\;(\R^3\se\ov{\Om})\times(0,T),\\
\ds\bm{n}\times\bm E=\bm{0}&\on\;\;\;\G\times(0,T),\\
\ds\bm E(x,0)=\bm H(x,0)=\bm{0}&\gin\;\;\;\R^3\se\ov{\Om},\\
\ds\hat{x}\times\left(\frac{\pa\bm E}{\pa t}\times\hat{x}\right)+\hat{x}\times\frac{\pa\bm{H}}{\pa t}
=o\left(\frac1{|x|}\right)\;\;&\text{as}\;\;\;|x|\rightarrow\infty,\;\;t\in(0,T).
\end{cases}
\en
Here, $\Om\subset\R^3$ is a bounded domain with Lipschitz boundary $\G$, $\bm E$ and $\bm H$ denote
the electric and magnetic fields, respectively, and $\hat{x}:={x}/{|x|}$.
The electric permittivity $\vep$ and the magnetic permeability $\mu$ are assumed to be positive constants
in this paper. The current density $\bm J$ is assumed to be compactly supported in the ball
$B_R:=\{x\in\R^3:|x|<R\}$ with boundary $\G_R$ for some $R>0$.

Define the time-domain electric-to-magnetic (EtM) Calder\'on operator $\mathscr{T}$ by
\be\label{tbc}
\mathscr{T}[\bm E_{\G_R}]=\bm H\ti\hat{x}\;\;\;\;\on\;\;\;\G_R\times(0,T),
\en
which is called the transparent boundary condition (TBC). Then, by using (\ref{tbc})
the scattering problem (\ref{2.1}) can be reduced into an equivalent initial-boundary value problem
in a bounded domain $\Om_R:=B_R\se\ov{\Om}$:
\be\label{2.2}
\begin{cases}
\ds\nabla\times\bm{E}+\mu\frac{\pa\bm{H}}{\pa t}=\0&\gin\;\;\;\Om_R\times(0,T),\\
\ds\nabla\times\bm{H}-\vep\frac{\pa\bm{E}}{\pa t}=\bm J&\gin\;\;\;\Om_R\times(0,T),\\
\ds\bm{n}\times\bm E=\bm{0}&\on\;\;\;\G\times(0,T),\\
\ds\bm E(x,0)=\bm H(x,0)=\bm{0}&\gin\;\;\;\Om_R,\\
\ds\mathscr{T}[\bm E_{\G_R}]=\bm H\times\hat{x}&\on\;\;\;\G_R\times(0,T).
\end{cases}
\en

In what follows, we will give a representation of the operator $\mathscr{T}$ together with its important
properties (see \cite{chen2008} for details).
Since $\bm J$ is supported in $B_R$, then, by taking the Laplace transform of (\ref{2.1}) with respect
to $t$ we obtain that
\be\label{2.3}
\nabla\times\ch{\bm{E}}+\mu s\ch{\bm{H}}&=&\bm{0}\;\;\;\gin\;\;\;\R^3\se\ov{B}_R,\\ \label{2.3a}
\nabla\times\ch{\bm{H}}-\vep s\ch{\bm{E}}&=&\bm{0}\;\;\;\gin\;\;\;\R^3\se\ov{B}_R,\\ \label{2.3b}
\hat{x}\times(\ch{\bm{E}}\times\hat{x})+\hat{x}\times\ch{\bm{H}}
&=&o\left(\frac1{|x|}\right)\;\;\;\text{as}\;\;\;|x|\rightarrow\infty.
\en
Let $\bm\lambda=\hat{x}\times\ch{\bm{E}}|_{G_R}$ and
let $\mathscr{B}:H^{-1/2}(\Curl,\G_R)\to H^{-1/2}(\Dive,\G_R)$ be the EtM Calder\'on operator
in $s$-domain defined by
\ben
\mathscr{B}[\bm\la\times\hat{x}]=\ch{\bm{H}}\times\hat{x}\;\;\;\on\;\;\;\G_R.
\enn
Then $\mathscr{T}=\mathscr{L}^{-1}\circ\mathscr{B}\circ\mathscr{L}$.
We now derive a representation of the operator $\mathscr{B}$. To this end,
denote by $\{\bm e_r,\bm e_{\theta},\bm e_{\phi}\}$ the unit vectors of the spherical coordinates
$(r,\theta,\phi)$:
\begin{equation*}
\begin{aligned}
\bm{e}_{r}&=(\sin\theta\cos \phi,\sin\theta\sin\phi,\cos\theta)^T, \\
\bm{e}_{\theta} &=(\cos\theta\cos\phi,\cos\theta\sin\phi,-\sin\theta)^T, \\
\bm{e}_{\phi} &=(-\sin\phi,\cos\phi,0)^T.
\end{aligned}
\end{equation*}
Let $\{Y_n^m(\hat{x}), m=-n,...n, n=1,2,...\}$ be the spherical harmonics forming a complete orthonormal
basis of $L^2(\mat{S}^2)$ and satisfying
\ben
\Delta_{\mat{S}^2}Y_n^m(\hat{x})+n(n+1)Y_n^m(\hat{x})=0,
\enn
where
\ben
\Delta_{\mat{S}^2}:=\frac{1}{\sin\theta}\frac{\pa}{\pa\theta}\left(\sin\theta\frac{\pa}{\pa\theta}\right)
+\frac{1}{\sin^2\theta}\frac{\pa^2}{\pa\phi^2}.
\enn
Let the vector spherical harmonics be denoted by
\ben
\bm U_n^m=\frac{1}{\sqrt{n(n+1)}}\na_{\G_R}Y_n^m,\;\;\;\bm V_n^m=\hat{x}\ti\bm U_n^m,
\enn
where
\ben
\na_{\G_R}:=\frac{1}{R}\left[\frac{\pa}{\pa\theta}\bm e_{\theta}
+\frac{1}{\sin\theta}\frac{\pa}{\pa\phi}\bm e_{\phi}\right]=\frac{1}{R}\na_{\mat{S}^2}.
\enn
Then $\{\bm{U}_n^m,\bm{V}_n^m, m=-n,...n, n=1,2,...\}$ forms a complete orthonormal basis of
$L_t^2(\G_R):=\{\bm u\in L^2(\G_R)^3:\;\bm u\cdot\hat{x}=0\;\;\on\;\;\G_R\}$.

For any $\bm\la\times\hat{x}=\sum_{n=1}^{\infty}\sum_{m=-n}^n\left[a_n^m\bm U_n^m(\hat{x})
+b_n^m\bm V_n^m(\hat{x})\right]$ on $\G_R$, we have that for $r=|x|\geq R$,
\ben
\ch{\bm E}(r,\hat{x})&=&\sum_{n=1}^{\ify}\sum_{m=-n}^n\left[\frac{Ra_n^mz_n^{(1)}(kr)}{rz_n^{(1)}(kR)}
\bm U_n^m+\frac{b_n^mh_n^{(1)}(kr)}{h_n^{(1)}(kR)}\bm V_n^m\right.\\
&&\qquad\qquad\qquad\left.+\frac{Ra_n^m\sqrt{n(n+1)}h_n^{(1)}(kr)}{rz_n^{(1)}(kR)}Y_n^m\hat{x}\right],\\
\ch{\bm H}(r,\hat{x})&=&\sum_{n=1}^{\ify}\sum_{m=-n}^n\left[\frac{b_n^mz_n^{(1)}(kr)}{\mu srh_n^{(1)}(kR)}
\bm U_n^m-\frac{\vep sRa_n^mh_n^{(1)}(kr)}{z_n^{(1)}(kR)}\bm V_n^m\right.\\
&&\qquad\qquad\qquad\left.-\frac{b_n^m\sqrt{n(n+1)}h_n^{(1)}(kr)}{\mu srh_n^{(1)}(kR)}Y_n^m\hat{x}\right],
\enn
which is the solution of the exterior problem (\ref{2.3})-(\ref{2.3b}) satisfying
that $\gamma_T\ch{\bm E}=\bm\la\times\hat{x}\;\;\on\;\;\Gamma_R$,
where $k=i\sqrt{\vep\mu}s$ with $\I(k)>0$, $h_n^{(1)}(z)$ is the spherical
Hankel function of the first kind of order n and $z_n^{(1)}(z)=h_n^{(1)}(z)+zh_n^{(1)\prime}(z).$
A simple calculation gives
\be\label{setm}
\mathscr{B}[\bm\la\times\hat{x}]=\ch{\bm{H}}\times\hat{x}\big|_{\G_R}
=-\sum_{n=1}^{\infty}\sum_{m=-n}^n\Big[\frac{\vep sRa_n^mh_n^{(1)}(kR)}{z_n^{(1)}(kR)}\bm U_n^m
+\frac{b_n^mz_n^{(1)}(kR)}{\mu sRh_n^{(1)}(kR)}\bm V_n^m\Big].\;\;\qquad
\en

We have the following important results on the continuity and coercivity of the operator $\mathscr{B}$
(see \cite[Theorem 9.21]{Monk2003} and \cite[Lemma 2.5]{chen2008}).

\begin{lemma}\label{lem2.1}
For each $s\in \C_+$, $\mathscr{B}:H^{-1/2}(\Curl,\G_R)\to H^{-1/2}(\Dive,\G_R)$ is bounded
with the estimate
\be\label{B-bound}
\|\mathscr{B}[\bm\la\times\hat{x}]\|^2_{H^{-1/2}(\Dive,\G_R)}
\lesssim (|s|^2+|s|^{-2})\|\bm\la\times\hat{x}\|^2_{H^{-1/2}(\Curl,\G_R)}.
\en
Further, we have
\ben
\Rt\langle\mathscr{B}\bm\om,\bm\om\rangle_{\G_R}\geq 0\;\;\;
\text{for any}\;\;\bm\om\in H^{-1/2}(\Curl,\G_R),
\enn
where $\langle\cdot\rangle_{\G_R}$ denotes the dual product between $H^{-1/2}(\Dive,\G_R)$
and $H^{-1/2}(\Curl,\G_R)$.
\end{lemma}

\begin{proof}
The boundedness and coercivity of the operator $\mathscr{B}$ have been proved in
\cite{Monk2003,chen2008} (see \cite[Theorem 9.21]{Monk2003} and \cite[Lemma 2.5]{chen2008}).
Here, we only prove the estimate (\ref{B-bound}) with the explicit dependence on $s$
which will be needed in Section \ref{sec5}.

By (\ref{setm}) and the definition of the norm of $H^{-1/2}(\Dive,\G_R)$ and $H^{-1/2}(\Curl,\G_R)$
we have
\ben
&&\|\mathscr{B}[\bm\la\ti\hat{x}]\|^2_{H^{-1/2}(\Dive,\G_R)}\\
&&\qquad=\sum_{n=1}^{\infty}\sum_{m=-n}^{n}\left[\sqrt{n(n+1)}
\left|\frac{\vep sRa_n^mh_n^{(1)}(kR)}{z_n^{(1)}(kR)}\right|^2
+\frac{1}{\sqrt{n(n+1)}}\left|\frac{b_n^mz_n^{(1)}(kR)}{\mu sRh_n^{(1)}(kR)}\right|^2\right].
\enn
By \cite[Lemma C.3]{LLA2015}, there exist two positive constants $C_1$ and $C_2$ such that
\ben
C_1 n\leq\left|\frac{z_n^{(1)}(kR)}{h_n^{(1)}(kR)}\right|\leq C_2 n.
\enn
Then it follows that
\ben
\|\mathscr{B}[\bm\la\times\hat{x}]\|^2_{H^{-1/2}(\Dive,\G_R)}
&\lesssim &\sum_{n=1}^{\infty}\sum_{m=-n}^n\left[|s|^2\frac{1}{\sqrt{n(n+1)}}|a_n^m|^2
+|s|^{-2}\sqrt{n(n+1)}|b_n^m|^2\right]\\
&\lesssim & (|s|^2+|s|^{-2})\|\bm\la\times\hat{x}\|^2_{H^{-1/2}(\Curl,\G_R)}.
\enn
The proof is thus complete.
\end{proof}

By Lemma \ref{lem2.1} and the Parseval identity, the coercivity of the time-domain EtM Calder\'on
operator $\mathscr{T}$ follows easily.

\begin{lemma}\label{lem2.2}
Given $t\geq 0$ and vector $\bm\om\in L^2\left(0,t;H^{-1/2}(\Curl,\G_R)\right)$, it follows that
\ben
\Rt\int_0^t\int_{\G_R}\mathscr{T}[\bm\om]\cdot\bar{\bm\om}d\g d\tau\geq 0.
\enn
\end{lemma}

\begin{proof}
Let $\wid{\bm\om}$ be the extension of $\bm\om$ by 0 with respect to $\tau$, that is,
$\wid{\bm\om}$ vanishes outside $[0,t]$. Combining the Parseval identity (\ref{A.5})
and Lemma {\ref{lem2.1}}, we have that for any $s_1>0$,
\ben
\Rt\int_0^te^{-2s_1\tau}\int_{\G_R}\mathscr{T}[\bm\om]\cdot\bar{\bm\om}d\g d\tau
&=&\Rt\int_{\G_R}\int_0^{\infty}e^{-2s_1\tau}\mathscr{T}[\wid{\bm\om}]\cdot\bar{\wid{\bm\om}}d\tau d\g\\
&=&\frac{1}{2\pi}\int_{-\infty}^{\infty}\Rt\langle\mathscr{B}[\ch{\wid{\bm\om}}],
\ch{\wid{\bm\om}}\rangle_{\G_R}ds_2\ge 0.
\enn
Taking the limit $s_1\rightarrow 0$ in the above inequality gives the required result.
\end{proof}

The well-posedness and stability of solutions of the scattering problem (\ref{2.2}) follow directly
from \cite[Theorem 3.1]{chen2008}.
Precisely, if $\bm J\in H^1(0,T;L^2(\Om_R)^3),\;\bm J|_{t=0}=0$ and $\bm J$ is extended so that
\ben
\bm J\in H^1(0,\infty;L^2(\Om_R)^3),\;\;\|\bm J\|_{H^1(0,\infty;L^2(\Om_R)^3)}
\le C\|\bm J\|_{H^1(0,T;L^2(\Om_R)^3)},
\enn
then we have
\ben
&\bm E\in L^2\left(0,T;H_{\G}(\curl,\Om_R)\right)\cap H^1\left(0,T;L^2(\Om_R)^3\right),\\
&\bm H\in L^2\left(0,T;H_{\G}(\curl,\Om_R)\right)\cap H^1\left(0,T;L^2(\Om_R)^3\right).
\enn
In particular, $\mathscr{T}[\bm E_{\G_R}]\in L^2\left(0,T;H^{-1/2}(\Dive,\G_R)\right)$.

To simplify the proof of the convergence of the PML method, we assume in the rest of this paper that
\be\label{assumption}
\bm J\in H^7(0,T;L^2(\Om_R)^3),\;\;\pa_t^j\bm J|_{t=0}=0,\;\;j=0,1,2,3,4,5,6
\en
and that $\bm J$ is extended so that
\be\label{assumption1}
\bm J\in H^7(0,\infty;L^2(\Om_R)^3),\;\;\|\bm J\|_{H^7(0,\infty;L^2(\Om_R)^3)}
\le C\|\bm J\|_{H^7(0,T;L^2(\Om_R)^3)}.
\en
Note that, under the assumption (\ref{assumption}), the differentiability with respect to $t$ of the
solution $(\bm E,\bm H)$ can be improved to the same order as $\bm J$, which can be easily verified
by using the Maxwell equations.

\section{The time-domain PML problem}\label{sec4}
\setcounter{equation}{0}

In this section, we first derive the time-domain PML formulation for the electromagnetic scattering problem
and then establish the well-posedness and stability of the PML problem by using the Laplace transform
and the energy method. Further, we prove the exponential convergence of the time-domain PML method.
\begin{figure}[!htbp]
\setcounter{subfigure}{0}
  \centering
  \includegraphics[width=2.4in]{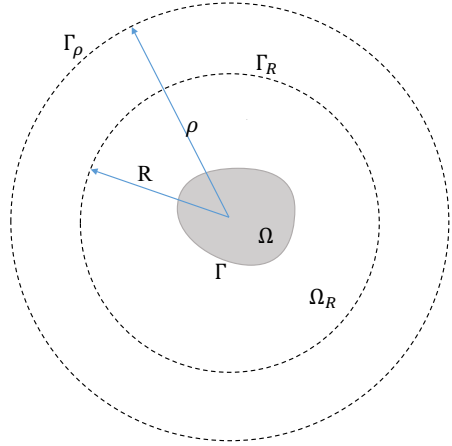}
 \caption{Geometric configuration of the PML layer}\label{geometry}
\end{figure}

\subsection{The PML problem and its well-posedness}

For $\rho>R$ let $\Om^{\PML}:=B_{\rho}\se \ov{B}_R=\{x\in\R^3: R<|x|<\rho\}$ denote the PML layer
with thickness $d:=\rho-R$, surrounding the bounded domain $\Om_R$. Denote by
$\Om_{\rho}:=B_{\rho}\se\ov{\Om}$ the truncated PML domain with the exterior boundary
$\G_{\rho}:=\{x\in\R^3:|x|=\rho\}$. See Figure \ref{geometry} for the geometry of the PML problem.
For $x=(x_1,x_2,x_3)^T\in\R^3$ consider the spherical coordinates
\ben
x_1=r\sin\theta\cos\phi,\;\;x_2=r\sin\theta\sin\phi,\;\;
x_3=r\cos\theta
\enn
with $r=|x|$ and the Euler angle $(\theta,\phi)$.

Now, let $s_1>0$ be an arbitrarily fixed parameter and let us define the PML medium property as
\ben
\al(r)=1+s_1^{-1}\sig(r),\;\;\;r=|x|,
\enn
where
\be\label{3.1}
\sig(r)=\begin{cases}
\ds 0,&0\le r\le R,\\
\ds \sig_0\Big(\frac{r-R}{\rho-R}\Big)^m,& R\le r\le\rho,\\
\ds\sig_0,&\rho\le r<\infty
\end{cases}
\en
with positive constant $\sig_0$ and integer $m\geq 1$.
In what follows, we will take the real part of the Laplace transform variable $s\in\C_+$ to be
$s_1$, that is, $\Rt(s)=s_1$.

In the rest of this paper, we always make the following assumptions on the thickness $d$
which are reasonable in our model:
\be\label{thick}
d\geq 1\;\;\;{\rm and}\;\;\;\rho\leq C_0d
\en
for some fixed generic constant $C_0$.

We will derive the PML equations by using the technique of change of variables. To this end, we
introduce the real stretched radius $\wid{r}$:
\be\label{3.1a}
\wid{r}=\int_0^r\al(\tau)d\tau=r\beta(r),
\en
where $\ds\beta(r)=\frac{1}{r}\int_0^r\al(\tau)d\tau$. For the Cartesian coordinates $x=(x_1,x_2,x_3)^T$,
the corresponding change of variables is $\wid{x}=(\wid{x}_1,\wid{x}_2,\wid{x}_3)^T$ with
\be\label{3.1b}
\wid{x}_1=\wid{r}\sin\theta\cos\phi,\;\;\wid{x}_2=\wid{r}\sin\theta\sin\phi,\;\;
\wid{x}_3=\wid{r}\cos\theta,
\en
where $\wid{r}$ denotes the real stretched radius of $r=|x|$ defined by (\ref{3.1a}).

To derive the PML equations, we introduce, respectively, the Maxwell single- and double-layer potentials
\ben
\bm\Psi_{{\rm SL}}(\bm q)=\int_{\G_R}\mathbb{G}^T(s,x,y)\bm q(y)d\g(y),\quad
\bm\Psi_{{\rm DL}}(\bm p)=\int_{\G_R}(\curl_y\mathbb{G})^T(s,x,y)\bm p(y)d\g(y),
\enn
where $\bm p=\g_t(\ch{\bm E})$ and $\bm q=\g_t(\curl\;\ch{\bm E})$ are the Dirichlet trace
and Neumann trace of the solution on $\G_R$, and $\mathbb{G}$ is the dyadic Green's function for
Maxwell's equations in the free space defined as a matrix function (see \cite[(12.1)]{Monk2003}):
\ben
\mathbb{G}(s,x,y)=\Phi_s(x,y)\mathbb{I}+\frac{1}{k^2}\nabla_y\nabla_y\Phi_s(x,y),\;\;\; x\neq y.
\enn
Hereafter, $s\in\C_+$ with $\Rt(s)=s_1$, $\mathbb{I}$ is the $3\times 3$ identity matrix,
$\Phi_s(x,y)$ is the fundamental solution of
the Helmholtz equation with complex wave number $k=i\sqrt{\vep\mu}s$ defined by
\be\label{3.1e}
\Phi_s(x,y)=\frac{e^{ik|x-y|}}{4\pi|x-y|}=\frac{e^{-\sqrt{\vep\mu}s|x-y|}}{4\pi|x-y|},
\en
and $\nabla_y\nabla_y\Phi_s(x,y)$ is the Hessian matrix of $\Phi_s(x,y)$ with its $(l,m)$th element
\be\label{3.1f}
(\nabla_y\nabla_y\Phi_s(x,y))_{l,m}=\frac{\pa^2\Phi_s(x,y)}{\pa_{y_l}\pa_{y_m}},\quad 1\leq l,m\leq 3.
\en
Then the solution of the exterior problem (\ref{2.3})-(\ref{2.3b}) is given by the integral
representation (see \cite[Theorem 12.2]{Monk2003})
\be\label{3.1c}
\ch{\bm E}(x)=-\bm\Psi_{{\rm SL}}(\bm q)(x)-\bm\Psi_{{\rm DL}}(\bm p)(x),\;\;\;
\ch{\bm H}(x)=-(\mu s)^{-1}\curl\;\ch{\bm E}(x).
\en

Let
\be\label{3.2}
\rho_s(\wid{x},y)=s|\wid{x}-y|=\left[s^2\left[(\wid{x}_1-y_1)^2+(\wid{x}_2-y_2)^2
+(\wid{x}_3-y_3)^2\right]\right]^{1/2}
\en
be the complex distance and let us define the stretched fundamental solution
\be\label{3.3}
\wid{\Phi}_s(x,y)=\frac{e^{-\sqrt{\vep\mu}\rho_s(\wid{x},y)}}{4\pi\rho_s(\wid{x},y)s^{-1}},
\en
where $z^{1/2}$ denotes the analytic branch of $\sqrt{z}$ satisfying that $\Rt(z^{1/2})>0$
for any $z\in\C\se(-\infty,0]$.

Now, for $x\in\R^3\se\ov{B}_R$ define the stretched single- and double-layer potentials
\ben
\wid{\bm\Psi}_{{\rm SL}}(\bm q)=\int_{\G_R}\wid{\mathbb{G}}^T(s,x,y)\bm q(y)d\g(y),\quad
\wid{\bm\Psi}_{{\rm DL}}(\bm p)=\int_{\G_R}(\curl_y\wid{\mathbb{G}})^T(s,x,y)\bm p(y)d\g(y),
\enn
where
\be\label{3.3a}
\wid{\mathbb{G}}(s,x,y)=\wid{\Phi}_s(x,y)\mathbb{I}
+\frac{1}{k^2}\nabla_y\nabla_y\wid{\Phi}_s(x,y),\;\;\;x\neq y,\;\;\;k=i\sqrt{\vep\mu}s.
\en
For any $\bm p\in H^{-1/2}(\Dive,\G_R)$ and $\bm q\in H^{-1/2}(\Dive,\G_R)$, let
\be\label{3.4}
\mathbb{E}(\bm p,\bm q)=-\wid{\bm\Psi}_{{\rm SL}}(\bm q)(x)-\wid{\bm\Psi}_{{\rm DL}}(\bm p)(x)
\en
denote the PML extensions in the $s$-domain of $\bm p$ and $\bm q$. Now, let
\be\label{3.4'}
\ch{\wid{\bm E}}(x)=\mathbb{E}(\g_t(\ch{\bm E}),\g_t(\curl\;\ch{\bm E})),\;\;\;
\ch{\wid{\bm H}}(x)=-(\mu s)^{-1}\wid{\curl}\;\ch{\wid{\bm E}}(x)
\en
be the PML extensions of $\g_t(\ch{\bm E})$ and $\g_t(\curl\;\ch{\bm E})$ on $\G_R$.
Then the stretched $\curl$ operator in the spherical coordinates is defined by
\ben
\wid{\curl}\;\bm u&=&\wid{\nabla}\times\bm u\\
&:=&\frac{1}{\wid{r}\sin\theta}\left[\frac{\pa(\sin\theta u_{\phi})}{\pa\theta}
    -\frac{\pa u_{\theta}}{\pa\phi}\right]\bm e_r
    +\left[\frac{1}{\wid{r}\sin\theta}\frac{\pa u_{r}}{\pa\phi}
    -\frac{1}{\wid{r}}\frac{\pa(\wid{r}u_{\phi})}{\pa\wid{r}}\right]\bm e_{\theta}\\
&&+\frac{1}{\wid{r}}\left[\frac{\pa(\wid{r}u_{\theta})}{\pa\wid{r}}
  -\frac{\pa u_r}{\pa{\theta}}\right]\bm e_{\phi}
\enn
with $u_r=\bm u\cdot\bm e_r$, $u_{\theta}=\bm u\cdot\bm e_{\theta}$ and $u_{\phi}=\bm u\cdot\bm e_{\phi}$
for any vector $\bm u$. It is easy to verify that
\ben
\wid{\nabla}\times\bm u=A(r)\nabla\times B(r)\bm u,
\enn
where $A(r)=\diag\{\beta^{-2},\al^{-1}\beta^{-1},\al^{-1}\beta^{-1}\}$
and $B(r)=\diag\{\al,\beta,\beta\}$.

It is clear that $\ch{\wid{\bm E}}$ and $\ch{\wid{\bm H}}$ satisfy
\be\label{3.4a}
\wid{\nabla}\times\ch{\wid{\bm{E}}}+\mu s\ch{\wid{\bm{H}}}=\0,\;\;\;
\wid{\nabla}\times\ch{\wid{\bm{H}}}-\vep s\ch{\wid{\bm{E}}}=\0\;\;\;\gin\;\;\;\R^3\se\ov{B}_R.
\en
Taking the inverse Laplace transform of (\ref{3.4a}) gives
\be\label{3.4b}
\wid{\nabla}\times\wid{\bm{E}}+\mu\pa_t\wid{\bm{H}}=\0,\;\;\;
\wid{\nabla}\times\wid{\bm{H}}-\vep\pa_t\wid{\bm{E}}=\0\;\;\;\gin\;\;\;(\R^3\se\ov{B}_R)\times(0,T).
\en
Define
\ben
(\bm E^{{\rm PML}},\bm H^{{\rm PML}}):=B(r)(\wid{\bm E},\wid{\bm H}).
\enn
Since $\ch{\wid{\bm E}}$ and $\ch{\wid{\bm H}}$ decay exponentially for $\Rt(s)=s_1>0$ as $r\to\infty$,
then $\wid{\bm E}$ and $\wid{\bm H}$ and thus $\bm E^{{\rm PML}}$ and $\bm H^{{\rm PML}}$ would
decay for $r\to\infty$. Further, since $\sig(R)=0$, then $\al=\beta=1$ on $\G_R$ so that
$\bm{E}^{{\rm PML}}=\bm E$ and $\bm{H}^{{\rm PML}}=\bm H$ on $\G_R$. Thus,
$(\bm{E}^{{\rm PML}},\bm{H}^{{\rm PML}})$ can be viewed as the extension of the solution of
the problem (\ref{2.1}).
If we set $\bm{E}^{{\rm PML}}=\bm E$ and $\bm{H}^{{\rm PML}}=\bm H$ in $\Om_R\ti(0,T)$,
then $(\bm E^{{\rm PML}},\bm H^{{\rm PML}})$ satisfies the PML problem
\be\label{pml}
\begin{cases}
\ds\nabla\times\bm{E}^{{\rm PML}}+\mu(BA)^{-1}\frac{\pa\bm{H}^{{\rm PML}}}{\pa t}=\0&\gin\;\;
   (\R^3\se\ov{\Om})\times(0,T),\\
\ds\nabla\times\bm{H}^{{\rm PML}}-\vep(BA)^{-1}\frac{\pa\bm{E}^{{\rm PML}}}{\pa t}=\bm J&\gin\;\;
   (\R^3\se\ov{\Om})\times(0,T),\\
\ds\bm{n}\times\bm E^{{\rm PML}}=\0&\on\;\;\G\times(0,T),\\
\ds\bm E^{{\rm PML}}(x,0)=\bm H^{{\rm PML}}(x,0)=\0&\gin\;\;\R^3\se\ov{\Om}.
\end{cases}
\en
The truncated PML problem in the time domain is to find $(\bm E^{p},\bm H^{p})$,
which is an approximation to $(\bm E,\bm H)$ in $\Om_R$ such that
\be\label{tpml}
\begin{cases}
\ds\nabla\times\bm{E}^p+\mu(BA)^{-1}\frac{\pa\bm{H}^p}{\pa t}=\0&\gin\;\;\Om_{\rho}\times(0,T),\\
\ds\nabla\times\bm{H}^p-\vep(BA)^{-1}\frac{\pa\bm{E}^p}{\pa t}=\bm J&\gin\;\;\Om_{\rho}\times(0,T),\\
\ds\bm{n}\times\bm E^p=\0&\on\;\;\G\times(0,T),\\
\ds\hat{x}\times\bm E^p=\0&\on\;\;\G_{\rho}\times(0,T),\\
\ds\bm E^p(x,0)=\bm H^p(x,0)=\0&\gin\;\;\Om_{\rho}.
\end{cases}
\en
Note that $s_1$ appearing in the matrices $A$ and $B$ is an arbitrarily fixed, positive parameter,
as mentioned earlier at the beginning of this subsection. In the Laplace transform domain,
the transform variable $s\in\C_+$ is taken so that $\Rt(s)=s_1>0$, and in the subsequent study of
the well-posedness and convergence of the truncated PML problem (\ref{tpml}), we take $s_1=1/T$.

The well-posedness of the truncated PML problem (\ref{tpml}) will be proved by using the Laplace transform
and the variational method. To this end, we first take the Laplace transform of the problem (\ref{tpml})
with the transform variable $s\in\C_+$ satisfying that $\Rt(s)=s_1$ and then eliminate the magnetic field
$\ch{\bm H}^{p}$ to obtain that
\be\label{espml}
\begin{cases}
\ds\nabla\times\left[(\mu s)^{-1}BA\nabla\times\ch{\bm E}^{p}\right]+\vep s(BA)^{-1}\ch{\bm E}^{p}
=-\ch{\bm J}\;\;&\gin\;\;\Om_{\rho},\\
\ds\bm{n}\times\ch{\bm E}^{p}=\0&\on\;\;\G,\\
\ds\hat{x}\times\ch{\bm E}^{p}=\0&\on\;\;\G_{\rho}.
\end{cases}
\en
It is easy to derive the variational formulation of (\ref{espml}):
Find a solution $\ch{\bm E}^p\in H_0(\curl,\Om_{\rho})$ such that
\be\label{3.8}
a_p(\ch{\bm{E}}^p,\bm V)=-\int_{\Om_R}\ch{\bm J}\cdot\ov{\bm V}dx,\quad\forall\;
\bm V\in H_0(\curl,\Om_{\rho}),
\en
where the sesquilinear form $a_p(\cdot,\cdot)$ is defined by
\ben
a_p(\ch{\bm{E}}^p,\bm V)
=\int_{\Om_{\rho}}(\mu s)^{-1}BA(\nabla\times\ch{\bm{E}}^p)\cdot(\nabla\times\ov{\bm{V}})dx
+\int_{\Om_{\rho}}s\vep(BA)^{-1}\ch{\bm{E}}^p\cdot\ov{\bm{V}}dx.
\enn

From (\ref{3.1}) we know that $\bt(r)=1+s_1^{-1}\hat{\sig}(r)$, where
\be\label{3.8a}
\hat{\sig}(r)=\begin{cases}
\ds\frac{1}{r}\int_R^{r}\sig(\tau)d\tau=\frac{\sig_0}{m+1}\frac{r-R}{r}\Big(\frac{r-R}{\rho-R}\Big)^m
 \;\;\; &\text{for $\;R\le r\le\rho$},\\
\ds\frac{\sig_0[(m+1)r-m\rho-R]}{[(m+1)r]}\;\;\; &\mbox{for $\;r\ge\rho$}.
\end{cases}
\en
It is obvious that
\ben
0\leq\hat{\sig}\leq\sig\leq\sig_0\;\;\;\text{for}\;\;R\le r\le\rho.
\enn
Noting that $BA=\diag\{\al\bt^{-2},\al^{-1},\al^{-1}\}$, we have
\be\no
&&\Rt[a_p(\ch{\bm{E}}^p,\ch{\bm{E}}^p)]\\ \no
&&=\int_{\Om_{\rho}}\frac{s_1}{\mu|s|^2}\left\{\frac{1+s_1^{-1}\sig}{(1+s_1^{-1}\hat{\sig})^2}
|(\na\ti\ch{\bm{E}}^p)_r|^2+\frac{1}{1+s_1^{-1}\sig}|(\nabla\times\ch{\bm{E}}^p)_\theta|^2\right.\\ \no
&&\qquad\qquad\qquad\left.+\frac{1}{1+s_1^{-1}\sig}|(\nabla\times\ch{\bm{E}}^p)_\phi|^2\right\}dx\\ \no
&&+\int_{\Om_\rho}\vep s_1\left\{\frac{(1+s_1^{-1}\hat{\sig})^2}{1+s_1^{-1}\sig}|\ch{\bm{E}}^p_r|^2
+(1+s_1^{-1}\sig)|\ch{\bm{E}}^p_\theta|^2+(1+s_1^{-1}\sig)|\ch{\bm{E}}^p_\phi|^2\right\}dx\\ \label{coercivity}
&&\gtrsim\frac{1}{1+s_1^{-1}\sig_0}\frac{s_1}{|s|^2}\left(\|\nabla\times\ch{\bm{E}}^p\|_{L^2(\Om_\rho)^3}^2
+\|s\ch{\bm{E}}^p\|^2_{L^2(\Om_\rho)^3}\right),
\en
which yields the strict coercivity of $a_p(\cdot,\cdot)$.

\begin{lemma}\label{lem3.1}
The variational problem $(\ref{3.8})$ of the problem $(\ref{espml})$ has a unique solution
$\ch{\bm E}^p\in H_0(\curl,\Om_{\rho})$ for each $s\in C_+$ with $\Rt(s)=s_1>0$. Further, it holds that
\be\label{3.11}
\|\nabla\times\ch{\bm E}^p\|_{L^2(\Om_{\rho})^3}+\|s\ch{\bm E}^p\|_{L^2(\Om_{\rho})^3}
\les s_1^{-1}(1+s_1^{-1}\sig_0)\|s\ch{\bm J}\|_{L^2(\Om_R)^3}.
\en
\end{lemma}

\begin{proof}
The first part of the lemma follows easily from the Lax-Milgram theorem and
the strict coercivity of $a_p(\cdot,\cdot)$, while
the estimate (\ref{3.11}) follows from (\ref{3.8}), (\ref{coercivity}) and
the Cauchy-Schwartz inequality. This completes the proof.
\end{proof}

The well-posedness and stability of the PML problem (\ref{tpml}) can be easily established
by using Lemma \ref{lem3.1} and the energy method (cf. \cite[Theorem 3.1]{chen2008}).

\begin{theorem}\label{thm3.2}
Let $s_1=1/T$. Then the truncated PML problem (\ref{tpml}) in the time domain has a unique solution
$(\bm E^p(x,t),\bm H^p(x,t))$ with
\ben
&&\bm E^p\in L^2\big(0,T;H_0(\curl,\Om_{\rho})\big)\cap H^1\left(0,T;L^2(\Om_{\rho})^3\right),\\
&&\bm H^p\in L^2\big(0,T;H_0(\curl,\Om_{\rho})\big)\cap H^1\left(0,T;L^2(\Om_{\rho})^3\right)
\enn
and satisfying the stability estimate
\ben
&&\max\limits_{t\in[0,T]}\left[\|\pa_t\bm E^p\|_{L^2(\Om_{\rho})^3}
  +\|\nabla\times\bm E^p\|_{L^2(\Om_{\rho})^3}
  +\|\pa_t\bm H^p\|_{L^2(\Om_{\rho})^3}+\|\nabla\times\bm H^p\|_{L^2(\Om_{\rho})^3}\right]\\
&&\qquad\les (1+\sig_0T)^2\|\bm J\|_{H^1(0,T;L^2(\Om_R)^3)}.
\enn
\end{theorem}

We now prove the well-posedness and stability of the solution in the PML layer $\Om^{\rm PML}$
which is needed for the convergence analysis of the PML method. Consider the initial boundary
value problem in the PML layer:
\be\label{3.12}
\begin{cases}
\ds\nabla\times\bm u+\mu(BA)^{-1}\frac{\pa\bm{v}}{\pa t}=\0 &\gin\;\;\Om^{\rm PML}\times(0,T),\\
\ds\nabla\times\bm v-\vep(BA)^{-1}\frac{\pa\bm u}{\pa t}=\0 &\gin\;\;\Om^{\rm PML}\times(0,T),\\
\ds\hat{x}\times\bm u=\0 &\on\;\;\G_R\times(0,T),\\
\ds\hat{x}\times\bm u=\bm\xi &\on\;\;\G_{\rho}\times(0,T),\\
\ds\bm u(x,0)=\bm v(x,0)=\0 &\gin\;\;\Om^{\rm PML}.
\end{cases}
\en
Taking the Laplace transform of (\ref{3.12}) with $\Rt(s)=s_1$
with respect to $t$ and eliminating $\ch{\bm v}$ give
\be\label{3.13}
\begin{cases}
\ds\nabla\times\left[(\mu s)^{-1}BA\nabla\times\ch{\bm u}\right]
+\vep s(BA)^{-1}\ch{\bm u}=\0\;\;&\gin\;\;\;\Om^{\rm PML},\\ 
\ds\hat{x}\times\ch{\bm u}=\0\;\;&\on\;\;\;\G_R,\\  
\ds\hat{x}\times\ch{\bm u}=\ch{\bm\xi}\;\;&\on\;\;\;\G_{\rho}.
\end{cases}
\en
Define the sesquilinear form
$a^{\rm PML}: H_{\G_R}(\curl,\Om^{\rm PML})\times H_{\G_R}(\curl,\Om^{\rm PML})\rightarrow\C$:
\be\label{3.14}
a^{\rm PML}(\ch{\bm u},\bm V):=\int_{\Om^{\rm PML}}(\mu s)^{-1}BA
(\nabla\times\ch{\bm u})\cdot(\nabla\times\ov{\bm{V}})dx
+\int_{\Om^{\rm PML}}s\vep (BA)^{-1}\ch{\bm u}\cdot\ov{\bm{V}}dx.\no\\
\en
Then the variational formulation of (\ref{3.13}) is as follows:
Given $\ch{\bm\xi}\in H^{-1/2}(\Dive,\G_\rho)$, find $\ch{\bm u}\in H_{\G_R}(\curl,\Om^{\rm PML})$
such that $\hat{x}\times\ch{\bm u}=\ch{\bm\xi}\;\on\;\G_{\rho}$ and
\be\label{3.14a}
a^{\rm PML}(\ch{\bm u},\bm V)=0,\;\;\;\forall\;\bm V\in H_0(\curl,\Om^{\rm PML}).
\en
Arguing similarly as in proving (\ref{coercivity}), we obtain that
for any $\bm V\in H_0(\curl,\Om^{\rm PML})$,
\be\label{3.15}
\Rt\left[a^{\rm PML}(\bm V,\bm V\right]
\gtrsim \frac{1}{1+s_1^{-1}\sig_0}\frac{s_1}{|s|^2}\left[\|\nabla\times\bm V\|_{L^2(\Om^{\rm PML})^3}^2
+\|s\bm V\|_{L^2(\Om^{\rm PML})^3}^2\right].\;\;\;
\en
Assume that $\bm\xi$ can be extended to a function in $H^2(0,\infty;H^{-1/2}(\Dive,\G_{\rho}))$ such that
\be\label{3.20}
\|\bm\xi\|_{H^2(0,\infty;H^{-1/2}(\Dive,\G_{\rho}))}&\les\|\bm\xi\|_{H^2(0,T;H^{-1/2}(\Dive,\G_{\rho}))}.
\en

By the Lax-Milgram theorem together with (\ref{3.15}) we know that the variational problem
(\ref{3.14a}) has a unique solution and thus the PML system (\ref{3.12}) is well-posed
(cf. the proof of Theorem \ref{thm3.2}). We now have the following stability result for the solution
to the PML system (\ref{3.12}).

\begin{theorem}\label{thm3.3}
Let $s_1=1/T$ and let $(\bm u,\bm v)$ be the solution of $(\ref{3.12})$. Then
\be\label{3.15a}
&&\|\pa_t\bm u\|_{L^2(0,T;L^2(\Om^{\rm PML})^3)}+\|\nabla\times\bm u\|_{L^2(0,T;L^2(\Om^{\rm PML})^3)}\no\\
&&\qquad\qquad\les(1+\sig_0T)^2T\|\bm\xi\|_{H^2(0,T;H^{-1/2}(\Dive,\G_\rho))},\\ \label{3.15b}
&&\|\pa_t\bm v\|_{L^2(0,T;L^2(\Om^{\rm PML})^3)}+\|\nabla\times\bm v\|_{L^2(0,T;L^2(\Om^{\rm PML})^3)}\no\\
&&\qquad\qquad\les(1+\sig_0T)^3T\|\bm\xi\|_{H^2(0,T;H^{-1/2}(\Dive,\G_\rho))}.
\en
\end{theorem}

\begin{proof}
Let $\bm u_0\in H_{\G_R}(\curl,\Om^{\PML})$ be such that $\hat{x}\times\bm u_0=\ch{\bm\xi}$
on $\G_{\rho}$.
Then, by (\ref{3.14a}) we have $\bm\om:=\ch{\bm u}-\bm u_0\in H_0(\curl,\Om^{\PML})$ and
\be\label{3.15a+}
a^{\rm PML}(\bm\om,\bm V)=-a^{\rm PML}(\bm u_0,\bm V),\;\;\;\forall\;\bm V\in H_0(\curl,\Om^{\rm PML}).
\en
This, combined with (\ref{3.14})-(\ref{3.15}) and the Cauchy-Schwartz inequality, gives
\ben
&&\frac{1}{1+s_1^{-1}\sig_0}\frac{s_1}{|s|^2}\left(\|\nabla\times\bm\om\|_{L^2(\Om^{\rm PML})^3}^2
+\|s\bm\om\|_{L^2(\Om^{\rm PML})^3}^2\right)\\
&&\les{}\Rt\left[a^{\rm PML}(\bm\om,\bm\om)\right]\\
&&\les{}\frac{(1+s_1^{-1}\sig_0)}{|s|}\sqrt{1+|s|^2}\left(\|\nabla\times\bm\om\|_{L^2(\Om^{\rm PML})^3}^2
+\|s\bm\om\|_{L^2(\Om^{\rm PML})^3}^2\right)^{1/2}\|\bm u_0\|_{H(\curl,\Om^{\PML})},
\enn
so
\ben
\|\nabla\times\bm\om\|_{L^2(\Om^{\rm PML})^3}^2+\|s\bm\om\|_{L^2(\Om^{\rm PML})^3}^2
\les\frac{(1+s_1^{-1}\sig_0)^4|s|^2(1+|s|^2)}{s_1^2}\|\bm u_0\|^2_{H(\curl,\Om^{\PML})}.
\enn
This, together with the definition of $\bm\om$ and the Cauchy-Schwartz inequality, implies
\ben 
\|\nabla\times\ch{\bm u}\|_{L^2(\Om^{\rm PML})^3}+\|s\ch{\bm u}\|_{L^2(\Om^{\rm PML})^3}
\les\frac{(1+s_1^{-1}\sigma_0)^2|s|(1+|s|)}{s_1}\|\bm u_0\|_{H(\curl,\Om^{\PML})}.
\enn
By the trace theorem we have
\be\no
&&\|\nabla\times\ch{\bm u}\|_{L^2(\Om^{\rm PML})^3}+\|s\ch{\bm u}\|_{L^2(\Om^{\rm PML})^3}\\ \label{3.16}
&&\qquad\qquad\les{s_1^{-1}}{(1+s_1^{-1}\sigma_0)^2}|s|(1+|s|)\|\ch{\bm\xi}\|_{H^{-1/2}(\Dive,\G_\rho)}.
\en
By (\ref{3.16}) and the Parseval equality (\ref{A.5}) it follows that
\ben
&&\int_0^T\left(\|\nabla\times\bm u\|_{L^2(\Om^{\rm PML})^3}^2
+\|\pa_t\bm u\|_{L^2(\Om^{\rm PML})^3}^2\right)dt\\
&&\le e^{2s_1T}\int_0^\infty e^{-2s_1t}\left(\|\nabla\times\bm u\|_{L^2(\Om^{\rm PML})^3}^2
+\|\pa_t\bm u\|_{L^2(\Om^{\rm PML})^3}^2\right)dt\\
&&\les 2\pi e^{2s_1T}s_1^{-2}(1+s_1^{-1}\sigma_0)^4\int_{-\infty}^{+\infty}|s|^2(1+|s|^2)
\|\ch{\bm\xi}\|^2_{H^{-1/2}(\Dive,\G_\rho)}ds_2\\
&&= e^{2s_1T}s_1^{-2}(1+s_1^{-1}\sigma_0)^4\int_0^\infty e^{-2s_1t}
\left(\|\pa_t\bm\xi\|^2_{H^{-1/2}(\Dive,\G_\rho)}
+\|\pa_t^2\bm\xi\|^2_{H^{-1/2}(\Dive,\G_\rho)}\right)dt\\
&&\les e^{2s_1T}s_1^{-2}(1+s_1^{-1}\sigma_0)^4\int_0^T\left(\|\pa_t\bm\xi\|^2_{H^{-1/2}(\Dive,\G_\rho)}
+\|\pa_t^2\bm\xi\|^2_{H^{-1/2}(\Dive,\G_\rho)}\right)dt,
\enn
where we have used (\ref{3.20}) to get the last inequality.

The required estimate (\ref{3.15a}) then follows from the above inequality with $s_1^{-1}=T$.
The required inequality (\ref{3.15b}) follows from (\ref{3.15a}) and the Maxwell equations
in (\ref{3.12}). The proof is thus complete.
\end{proof}

\section{Exponential convergence of the PML method}\label{sec5}

In this section, we prove the exponential convergence of the PML method.
We first start with the following lemma which was proved in \cite[Lemma 4.1]{CW2012} for
the two-dimensional case. The three-dimensional case can be easily proved similarly.

\begin{lemma}\label{lem4.1}
For any $z_j=a_j+ib_j$ with $a_j,b_j\in\R$ such that $b_1^2+b_2^2+b_3^2>0$, $j=1,2,3$, we have
\ben
\Rt\left[(z_1^2+z_2^2+z_3^2)^{1/2}\right]\geq\frac{|a_1b_1+a_2b_2+a_3b_3|}{\sqrt{b_1^2+b_2^2+b_3^2}}.
\enn
\end{lemma}

The following lemma is useful in the proof of the exponential decay property of the stretched
fundamental solution $\wid{\Phi}_s(x,y)$.

\begin{lemma}\label{lem4.2}
Let $s=s_1+is_2$ with $s_1>0,\;s_2\in\R$. Then, for any $x\in\G_{\rho}$ and $y\in\G_R$,
the complex distance $\rho_s$ defined by (\ref{3.2}) satisfies
\ben
|\rho_s(\wid{x},y)/s|\geq d,\;\;\;\;\Rt[\rho_s(\wid{x},y)]\geq\rho\hat{\sig}(\rho),
\enn
where, by (\ref{3.8a}) $\hat{\sig}(\rho)$ is given as
\be\label{4.0}
\hat{\sig}(\rho)=\frac{1}{\rho}\int_R^{\rho}\sig(\tau)d\tau=\frac{\sig_0d}{\rho(m+1)}.
\en
\end{lemma}

\begin{proof}
For $x\in\G_{\rho}$ and $y\in\G_R$, write $\wid{x}=(\wid{x}_1,\wid{x}_2,\wid{x}_3)$
and $y=(y_1,y_2,y_3)$ in the spherical coordinates with
\ben
&&\wid{x}_1=\wid{\rho}\sin\theta_1\cos\phi_1,\;\;\;\wid{x}_2=\wid{\rho}\sin\theta_1\sin\phi_1,
\;\;\;\wid{x}_3=\wid{\rho}\cos\theta_1,\\
&&y_1=R\sin\theta_2\cos\phi_2,\;\;\;y_2=R\sin\theta_2\sin\phi_2,\;\;\;y_3=R\cos\theta_2,
\enn
where $\wid{x}$ is the stretched coordinates of $x=(x_1,x_2,x_3)$ and $\wid{\rho}$
denotes the real stretched radius of $\rho=|x|$ defined similarly as in (\ref{3.1b}).
Then, by the definition of the complex distance $\rho_s(\wid{x},y)$ (see (\ref{3.2})) we have
\ben
|\rho_s(\wid{x},y)/s|&=&|\wid{x}-y|=\sqrt{(\wid{x}_1-y_1)^2+(\wid{x}_2-y_2)^2+(\wid{x}_3-y_3)^2}\\
&=&\sqrt{\wid{\rho}^2+R^2-2\wid{\rho}R[\sin\theta_1\sin\theta_2\cos(\phi_1-\phi_2)
+\cos\theta_1\cos\theta_2]}\\
&\ge&\wid{\rho}-R\geq\rho-R.
\enn
In addition, by Lemma \ref{lem4.1} we know that
\ben
\Rt\left[\rho_s(\wid{x},y)\right]&=&\Rt\left[s^2\left((\wid{x}_1-y_1)^2+(\wid{x}_2-y_2)^2
+(\wid{x}_3-y_3)^2\right)\right]^{1/2}\\
&\ge&\frac{|s_1s_2(\wid{x}_1-y_1)^2+s_1s_2(\wid{x}_2-y_2)^2+s_1s_2(\wid{x}_3-y_3)^2|}
{\sqrt{s_2^2(\wid{x}_1-y_1)^2+s_2^2(\wid{x}_2-y_2)^2+s_2^2(\wid{x}_3-y_3)^2}}\\
&=& s_1\sqrt{(\wid{x}_1-y_1)^2+(\wid{x}_2-y_2)^2+(\wid{x}_3-y_3)^2}\\
&\ge& s_1(\wid{\rho}-R)\\
&\ge&\rho\hat{\sig}(\rho).
\enn
This completes the proof.
\end{proof}

The following lemma gives the estimates of the stretched dyadic Green's function $\wid{\mathbb{G}}$
of the PML equation which plays a key role in the convergence analysis of the PML method.

\begin{lemma}\label{lem4.3}
Assume that the conditions in $(\ref{thick})$ are satisfied. Then we have that
for $x\in\G_{\rho}$, $y\in\G_{R}$,
\be\label{4.1}
&&\left|\wid{\mathbb{G}}(s,x,y)\right|
\les s_1^{-2}d^{-1}(1+s_1^{-1}\sigma_0)^2e^{-\sqrt{\vep\mu}\rho\hat{\sigma}(\rho)},\\ \label{4.1a}
&&\left|\curl_{\wid{x}}\wid{\mathbb{G}}(s,x,y)\right|,\;\;
\left|\curl_y\wid{\mathbb{G}}(s,x,y)\right|
\les d^{-1}(1+|s|)(1+s_1^{-1}\sigma_0)e^{-\sqrt{\vep\mu}\rho\hat{\sigma}(\rho)},\\ \label{4.1b}
&&\left|\curl_{\wid{x}}\curl_{y}\wid{\mathbb{G}}(s,x,y)\right|,\;\;
\left|\curl_{y}\curl_{y}\wid{\mathbb{G}}(s,x,y)\right|\no\\
&&\hspace{4cm}\les (1+|s|^2)(1+s_1^{-1}\sigma_0)^2d^{-1}e^{-\sqrt{\vep\mu}\rho\hat{\sigma}(\rho)},\\ \label{4.1c}
&&\left|\curl_{\wid{x}}\curl_{y}\curl_{y}\wid{\mathbb{G}}(s,x,y)\right|
\les (1+|s|^3)d^{-1}(1+s_1^{-1}\sigma_0)^3e^{-\sqrt{\vep\mu}\rho\hat{\sigma}(\rho)},
\en
where $\wid{\mathbb{G}}$ is the stretched dyadic Green's function and $s=s_1+is_2\in\C_+$.
\end{lemma}

\begin{proof}
For $i,j,k=1,2,3$. By Lemma \ref{lem4.2} and the definition of the stretched fundamental
solution $\wid{\Phi}_s$ in (\ref{3.3}) we have
\be\label{4.2}
\left|\wid{\Phi}_s(x,y)\right|&=&
\frac{e^{-\sqrt{\vep\mu}\Rt[\rho_s(\wid{x},y)]}}{4\pi|\rho_s(\wid{x},y)/s|}
\le\frac{e^{-\sqrt{\vep\mu}\rho\hat{\sigma}(\rho)}}{4\pi d},\\ \label{4.2b}
-\frac{\pa\wid{\Phi}_s(x,y)}{\pa\wid{x}_j}&=&\frac{\pa\wid{\Phi}_s(x,y)}{\pa y_j}
=s\frac{\sqrt{\vep\mu}(\wid{x}_j-y_j)}{\rho_s(\wid{x},y)/s}\wid{\Phi}_s(x,y)
+\frac{(\wid{x}_j-y_j)}{[\rho_s(\wid{x},y)/s]^2}\wid{\Phi}_s(x,y)\no\\
&:=& sP_{1,j}^s+P_{0,j}^s.
\en
By the conditions in (\ref{thick}) we know that
\ben
|\wid{x}_j-y_j|\le|\wid{x}-y|\le\wid{\rho}+R
=\rho+R+s_1^{-1}\rho\hat{\sigma}(\rho)\les (1+s_1^{-1}\sigma_0)d,
\enn
and so
\be\label{4.2c}
|P_{l,j}^s|\les(1+s_1^{-1}\sigma_0)d^{-1}e^{-\sqrt{\vep\mu}\rho\hat{\sigma}(\rho)},\;\;\;l=0,1.
\en
For the second-order derivatives of $\wid{\Phi}_s$, we have
\be\label{4.2c+}
-\frac{\pa^2\wid{\Phi}_s(x,y)}{\pa \wid{x}_i\pa y_j}=\frac{\pa^2\wid{\Phi}_s(x,y)}{\pa y_i\pa y_j}=s^2Q_{2,ij}^s+sQ_{1,ij}^s+Q_{0,ij}^s,
\en
where
\ben
&&Q_{2,ij}=\frac{\sqrt{\vep\mu}(\wid{x}_j-y_j)}{\rho_s(\wid{x},y)/s}P_{1,i}^s,\\
&&Q_{1,ij}=\frac{(\wid{x}_j-y_j)}{[\rho_s(\wid{x},y)/s]^2}P_{1,i}^s
+\frac{\sqrt{\vep\mu}(\wid{x}_j-y_j)}{\rho_s(\wid{x},y)/s}P_{0,i}^s
+\frac{\sqrt{\vep\mu}[(\wid{x}_i-y_i)(\wid{x}_j-y_j)
-\delta_{i,j}\rho_s(\wid{x},y)/s]}{[\rho_s(\wid{x},y)/s]^2}\wid{\Phi}_s,\\
&&Q_{0,ij}=\frac{(\wid{x}_j-y_j)}{[\rho_s(\wid{x},y)/s]^2}P_{0,i}^s
+\frac{2(\wid{x}_i-y_i)(\wid{x}_j-y_j)
-\delta_{i,j}[\rho_s(\wid{x},y)/s]^2}{[\rho_s(\wid{x},y)/s]^4}\wid{\Phi}_s,
\enn
where $\delta_{i,j}$ denotes the Kronecker symbol. By (\ref{4.2}) and (\ref{4.2c}) it follows that
\be\label{4.2c++}
|Q_{l,ij}^s|\les(1+s_1^{-1}\sigma_0)^2d^{-1}e^{-\sqrt{\vep\mu}\rho\hat{\sigma}(\rho)},\;\;\;l=0,1,2.
\en
This, together with (\ref{4.2}), (\ref{4.2c+}) and the definition of $\wid{\mathbb{G}}$
in (\ref{3.3a}), implies (\ref{4.1}).

By noting the fact that the $\curl$ of the Hessian is zero, we know that the $\curl$ of the dyadic
Green function $\wid{\mathbb{G}}(s,x,y)$ only includes the $\curl$ of $\Phi_s(x,y)\mathbb{I}$ .
Thus (\ref{4.1a}) and (\ref{4.1b}) follow from (\ref{4.2b})-(\ref{4.2c}) and
(\ref{4.2c+})-(\ref{4.2c++}), respectively.

To prove (\ref{4.1c}), we also need the estimates for the third-order derivatives of $\wid{\Phi}_s$.
First we have
\ben
-\frac{\pa^3\wid{\Phi}_s(x,y)}{\pa\wid{x}_i\pa y_j\pa y_k}
=\frac{\pa^3\wid{\Phi}_s(x,y)}{\pa y_i\pa y_j\pa y_k}
=s^3R_{3,ijk}^s+s^2R_{2,ijk}^s+sR_{1,ijk}^s+R_{0,ijk}^s,
\enn
where, by a direct calculation, we can prove similarly as above that
\ben
|R_{l,ijk}^s|\les(1+s_1^{-1}\sigma_0)^3d^{-1}e^{-\sqrt{\vep\mu}\rho\hat{\sigma}(\rho)},\;\;\;l=0,1,2,3.
\enn
This, together with the definition of $\wid{\mathbb{G}}$ and the fact that $\curl$ of the Hessian is zero,
yields (\ref{4.1c}). The proof is thus complete.
\end{proof}

\begin{theorem}\label{thm5.4}
For any $\bm p\in H^{-1/2}(\Dive,\G_R)$ and $\bm q\in H^{-1/2}(\Dive,\G_R)$,
let $\mathbb{E}(\bm p,\bm q)$ be the PML extension in the $s$-domain defined in $(\ref{3.4})$.
Then, for any $x\in\Om^{\PML}$ we have
\be\label{4.2d}
&&|\mathbb{E}(\bm p,\bm q)(x)|\\ \no
&&\les s_1^{-2}d^{1/2}(1+s_1^{-1}\sigma_0)^2e^{-\sqrt{\vep\mu}\rho\hat{\sigma}(\rho)}
\left[(1+|s|)\|\bm q\|_{H^{-1/2}(\Dive,\G_R)}
+(1+|s|^2)\|\bm p\|_{H^{-1/2}(\Dive,\G_R)}\right]
\en
and
\be\label{4.2e}
&&|\curl_{\wid{x}}\;\mathbb{E}(\bm p,\bm q)(x)|\\ \no
&&\les d^{1/2}(1+s_1^{-1}\sigma_0)^3e^{-\sqrt{\vep\mu}\rho\hat{\sigma}(\rho)}
\left[(1+|s|^2)\|\bm q\|_{H^{-1/2}(\Dive,\G_R)}
+(1+|s|^3)\|\bm p\|_{H^{-1/2}(\Dive,\G_R)}
\right].
\en
\end{theorem}

\begin{proof}
Since $\g_T$ is a bounded operator, by Lemma \ref{lem4.3} we have
\ben
|\wid{\bm\Psi}_{{\rm SL}}(\bm q)(x)|&\le&\|\bm q\|_{H^{-1/2}(\Dive,\G_R)}\cdot
 \|\g_T\wid{\mathbb{G}}(s,x,\cdot)\|_{H^{-1/2}(\Curl,\G_R)}\\
&\les&\|\bm q\V_{H^{-1/2}(\Dive,\G_R)}\cdot\|\wid{\mathbb{G}}(s,x,\cdot)\|_{H(\curl,\Om_R)}\\
&\les& s_1^{-2}d^{1/2}(1+s_1^{-1}\sigma_0)^2e^{-\sqrt{\vep\mu}\rho\hat{\sigma}(\rho)}(1+|s|)
   \|\bm q\|_{H^{-1/2}(\Dive,\G_R)},\\
|\wid{\bm\Psi}_{{\rm DL}}(\bm p)(x)|&\le&\|\bm p\|_{H^{-1/2}(\Dive,\G_R)}\cdot
\|\g_T(\curl\;\wid{\mathbb{G}})(s,x,\cdot)\|_{H^{-1/2}(\Curl,\G_R)}\\
&\les& \|\bm p\|_{H^{-1/2}(\Dive,\G_R)}\cdot\|\curl\;\wid{\mathbb{G}}(s,x,\cdot)\|_{H(\curl,\Om_R)}\\
&\les& d^{1/2}(1+s_1^{-1}\sigma_0)^2e^{-\sqrt{\vep\mu}\rho\hat{\sigma}(\rho)}(1+|s|^2)
\|\bm p\|_{H^{-1/2}(\Dive,\G_R)}.
\enn
This together with (\ref{3.4}) gives (\ref{4.2d}).
The estimate (\ref{4.2e}) for $|\curl_{\wid{x}}\;\mathbb{E}(\bm p,\bm q)(x)|$
can be proved similarly. The proof is complete.
\end{proof}

We are now ready to prove the exponential convergence of the time-domain PML method,
as stated in the following theorem.

\begin{theorem}\label{thm5.5}
Let $(\bm E,\bm H)$ and $(\bm E^p,\bm H^p)$ be the solutions of the problems $(\ref{2.2})$
and $(\ref{tpml})$ with $s_1=1/T$, respectively.
If the assumptions $(\ref{assumption})$ and $(\ref{assumption1})$ are satisfied, then
\be\label{4.43}
&&\max_{0\le t\le T}(\|\bm E-\bm E^p\|_{L^2(\Om_R)^3}+\|\bm H-\bm H^p\|_{L^2(\Om_R)^3})\no\\
&&\qquad\qquad\les T^{9/2}d^2(1+\sigma_0T)^{9}e^{-\sigma_0d\sqrt{\vep\mu}/2}
\|\bm J\|_{H^7(0,T;L^2(\Om_R)^3)}.
\en
\end{theorem}

\begin{proof}
By (\ref{2.2}) and (\ref{tpml}) it follows that
\be\label{4.3}
&&\nabla\times(\bm{E}-\bm E^p)+\mu\frac{\pa(\bm{H}-\bm H^p)}{\pa t}
   =\0\;\;\;\;\gin\;\;\;\Om_R\times(0,T),\\ \label{4.4}
&&\nabla\times(\bm{H}-\bm H^p)-\vep\frac{\pa(\bm{E}-\bm E^p)}{\pa t}
   =\0\;\;\;\;\gin\;\;\;\Om_R\times(0,T).
\en
Multiplying both sides of (\ref{4.4}) by the complex conjugate of $\bm V\in H_{\G}(\curl,\Om_R)$
and integrating by parts, we obtain
\be\label{4.4a}
(\bm{H}-\bm H^p,\nabla\times\bm V)_{\Om_R}-\vep(\pa_t(\bm{E}-\bm E^p),\bm V)_{\Om_R}
-\langle\g_t(\bm{H}-\bm H^p),\g_T\bm V\rangle_{\G_R}=0.\;\;\quad
\en
Define
\be\label{4.4b}
\bm u:=\bm{E}-\bm E^p,\;\;\;\;\bm u^*:=\int_0^t\bm ud\tau.
\en
Taking $\bm V=\bm u$ in (\ref{4.4a}) and using (\ref{4.3}) and the TBC (\ref{tbc}), we obtain
\be\no
&&\mu(\bm{H}-\bm H^p,\pa_t(\bm H-\bm H^p))_{\Om_R}+\vep(\pa_t\bm u,\bm u)_{\Om_R}
+\langle\mathscr{T}[\bm u_{\G_R}],\g_T\bm u\rangle_{\G_R}\\ \label{4.5}
&&\hspace{4cm}\qquad=\langle\g_t\bm H^p-\mathscr{T}[\g_T\bm E^p],\g_T\bm u\rangle_{\G_R}.
\en
Now, from (\ref{4.3}) it follows that $\nabla\times\bm u^*=-\mu(\bm{H}-\bm H^p)$.
Thus taking the real part of both sides of (\ref{4.5}) leads to
\be\no
&&\frac{1}{2}\frac{d}{dt}\left(\mu^{-1}\|\nabla\times\bm u^*\|^2_{L^2(\Om_R)^3}
+\vep\|\bm u\|^2_{L^2(\Om_R)^3}\right)
+\Rt\langle\mathscr{T}[\bm u_{\G_R}],\g_T\bm u\rangle_{\G_R}\\ \label{4.5a}
&&\qquad=\Rt\langle\g_t\bm H^p-\mathscr{T}[\g_T\bm E^p],\bm u\rangle_{\G_R}.
\en
Define the Banach space
\ben
X(0,T;\Om_R):=\left\{\bm v\in L^{\infty}(0,T;L^2(\Om_R)^3),\;
\bm v^*=\int_0^t\bm vd\tau\in L^{\infty}(0,T;H(\curl,\Om_R))\right\}
\enn
with the norm
\ben
\|\bm v\|_{X(0,T;\Om_R)}=\sup_{0\le t\le T}\left[\|\bm v\|^2_{L^2(\Om_R)^3}
+\|\nabla\times\bm v^*\|^2_{L^2(\Om_R)^3}\right]^{1/2}.
\enn
Define further the Banach space
\ben
Y(0,T;\G_R):=\left\{\bm\om:\int_0^T\langle\bm\om,\bm v\rangle_{\G_R}dt<\infty,\;\;
\forall\;\bm v\in X(0,T;\Om_R)\right\}
\enn
with the norm
\ben
\|\bm\om\|_{Y(0,T;\G_R)}=\sup_{\bm v\in X(0,T;\Om_R)}
\frac{\left|\int_0^T\langle\bm\om,\bm v\rangle_{\G_R}dt\right|}{\|\bm v\|_{X(0,T;\Om_R)}}.
\enn
By (\ref{4.5a}) and Lemma \ref{lem2.2} we get
\be\label{4.32}
\|\nabla\times\bm u^*\|_{L^2(\Om_R)^3}+\|\bm u\|_{L^2(\Om_R)^3}
\les\left\|\g_t\bm H^p-\mathscr{T}[\g_T\bm E^p]\right\|_{Y(0,T;\G_R)}.
\en

For $\ch{\bm E}^p|_{\G_R}$ define its PML extension $\ch{\wid{\bm E}^p}$ in the $s$-domain to be
the solution of the exterior problem
\ben
\begin{cases}
\wid{\na}\times[(\mu s)^{-1}\wid{\na}\times\bm u]+\vep s\bm u=\0&\gin\;\;\;\R^3\se\ov{B}_{R},\\
\hat{x}\times\bm u=\hat{x}\times\ch{\bm E}^p&\on\;\;\;\G_{R},\\
\ds\hat{x}\times(\mu s\bm{u}\times\hat{x})-\hat{x}\times(\wid{\na}\times\bm{u})
=o\left(\frac{1}{|\wid{x}|}\right) &\text{as}\;\;\;|\wid{x}|\rightarrow\infty.
\end{cases}
\enn
By \cite[Theorem 12.2]{Monk2003} it is easy to see that $\ch{\wid{\bm E}^p}$ satisfies
the integral representation
\be\label{4.33}
\ch{\wid{\bm E}^p}=\mathbb{E}(\g_t(\ch{\bm E}^p),\g_t(\wid{\curl}\;\ch{\wid{\bm E}^p})).
\en
Define $\ch{\wid{\bm H}^p}:=-(\mu s)^{-1}\wid{\curl}\ch{\wid{\bm E}^p}$.
Then $(\ch{\wid{\bm E}^p},\ch{\wid{\bm H}^p})$ satisfies the stretched Maxwell
equations in (\ref{3.4a}) in $\R^3\se\ov{B}_{R}$.
It's worth noting that $\ch{\wid{\bm H}^p}$ is not the extension of $\ch{\bm H^p}|_{\G_R}$.
Now let
\ben
\wid{\bm E}^p=\mathscr{L}^{-1}(\ch{\wid{\bm E}^p}),\;\;\;\;
\wid{\bm H}^p=\mathscr{L}^{-1}(\ch{\wid{\bm H}^p}).
\enn
Then $(\wid{\bm E}^p,\wid{\bm H}^p)$ satisfies the Maxwell equations in (\ref{3.4b})
in $\R^3\se\ov{B}_{R}\times(0,T)$. Further, we can claim that
\be\label{tetm}
\mathscr{T}[\g_T\bm E^p]=\g_t(B\wid{\bm H}^p)\;\;\;\on\;\;\;\G_R\times(0,T).
\en
In fact, the tangential component $\g_T\ch{\bm E^{p}}=(\hat{x}\times\ch{\bm E^{p}})\times\hat{x}|_{\G_R}$
of the solution $\ch{\bm E^{p}}$ on $\G_R$ can be represented in terms of the complete orthonormal basis $\{\bm{U}_n^m,\bm{V}_n^m, m=-n,...n, n=1,2,...\}$ of $L_t^2(\G_R)$ as
\ben
\g_T\ch{\bm E^{p}}=\sum_{n=1}^{\infty}\sum_{m=-n}^n\left[\wid{a_n^m}\bm U_n^m(\hat{x})
+\wid{b_n^m}\bm V_n^m(\hat{x})\right]\;\;\;\on\;\;\;\G_R.
\enn
where $\wid{a_n^m}$ and $\wid{b_n^m}$ depend only on $s,k,R$.
Then, by the definition of the EtM operator $\mathscr{B}$ in (\ref{setm}) we have
\be\label{B-EtM}
\mathscr{B}[\g_T\ch{\bm E^{p}}]=-\sum_{n=1}^{\infty}\sum_{m=-n}^n
\left[\frac{\vep sR\wid{a_n^m}h_n^{(1)}(kR)}{z_n^{(1)}(kR)}\bm U_n^m
+\frac{\wid{b_n^m}z_n^{(1)}(kR)}{\mu sRh_n^{(1)}(kR)}\bm V_n^m\right].
\en
On the other hand, $(\ch{\wid{\bm E}^p},\ch{\wid{\bm H}^p})$ satisfies the exterior problem
\ben
&&\wid{\na}\times\ch{\wid{\bm E}^p}+\mu s\ch{\wid{\bm H}^p}=\0\;\;\;\gin\;\;\;\R^3\se\ov{B}_R,\\
&&\wid{\na}\ti \ch{\wid{\bm H}^p}-\vep s\ch{\wid{\bm E}^p}=\0\;\;\;\gin\;\;\;\R^3\se\ov{B}_R,\\
&&\hat{x}\times\ch{\wid{\bm E}^p}=\hat{x}\times\ch{\bm E^{p}}\;\;\;\on\;\;\;\G_R,\\
&&\ds\hat{x}\times(\ch{\wid{\bm E}^p}\times\hat{x})+\hat{x}\times\ch{\wid{\bm{H}}^p}
=o\left(\frac{1}{|\wid{x}|}\right)\;\;\;\;\text{as}\;\;\;|\wid{x}|\rightarrow\infty.
\enn
The solution of this exterior problem is given by
\ben
\ch{\wid{\bm E}^p}(\wid{r},\hat{x})&=&\sum_{n=1}^{\infty}\sum_{m=-n}^n
\left[\frac{R\wid{a_n^m}z_n^{(1)}(k\wid{r})}{\wid{r}z_n^{(1)}(kR)}\bm U_n^m
+\frac{\wid{b_n^m}h_n^{(1)}(k\wid{r})}{h_n^{(1)}(kR)}\bm V_n^m\right.\\
&&\qquad\qquad\qquad\left.+\frac{R\wid{a_n^m}\sqrt{n(n+1)}h_n^{(1)}(k\wid{r})}{\wid{r}z_n^{(1)}(kR)}
  Y_n^m\hat{x}\right],\\
\ch{\wid{\bm H}^p}(\wid{r},\hat{x})&=&\sum_{n=1}^{\infty}\sum_{m=-n}^n
\left[\frac{\wid{b_n^m}z_n^{(1)}(k\wid{r})}{\mu s\wid{r}h_n^{(1)}(kR)}\bm U_n^m
-\frac{\vep sR\wid{a_n^m}h_n^{(1)}(k\wid{r})}{z_n^{(1)}(kR)}\bm V_n^m\right.\\
&&\qquad\qquad\qquad\left.-\frac{\wid{b_n^m}\sqrt{n(n+1)}h_n^{(1)}(k\wid{r})}{\mu s\wid{r}h_n^{(1)}(kR)}
  Y_n^m\hat{x}\right]
\enn
for $r\geq R$. Since $\wid{r}=R,\,B=I$ on $\G_R$, we have
\be\no
\g_t(B\ch{\wid{\bm H}^p})&=&B\ch{\wid{\bm H}^p}\times\hat{x}\big|_{\G_R}\\ \label{B-EtM+}
&=&-\sum_{n=1}^{\infty}\sum_{m=-n}^n\left[\frac{\vep sR\wid{a_n^m}h_n^{(1)}(kR)}{z_n^{(1)}(kR)}\bm U_n^m
+\frac{\wid{b_n^m}z_n^{(1)}(kR)}{\mu sRh_n^{(1)}(kR)}\bm V_n^m\right],
\en
where we have used the fact that $\bm U_n^m\times\hat{x}=-\bm V_n^m$, $\bm V_n^m\times\hat{x}=\bm U_n^m$.
By (\ref{B-EtM}) and (\ref{B-EtM+}) we have $\mathscr{B}[\g_T\ch{\bm E^{p}}]=\g_t(B\ch{\wid{\bm H}^p})$
$\on\;\;\G_R$. Taking the inverse Laplace transform of this equation gives the desired equality (\ref{tetm}).

Now, by (\ref{4.32}) and (\ref{tetm}), and since any function $\bm v\in X(0,T;\Om_R)$ can be extended
into $\Om^{{\rm PML}}\times(0,T)$ (denoted again by $\bm v$) such that
\ben
\g_t\bm v=0\;\;\;\on\;\;\;\G_{\rho}\;\;\;\;{\rm and}\;\;\;\;
\|\bm v\|_{X(0,T;\Om^{{\rm PML}})}\le C\|\bm v\|_{X(0,T;\Om_R)},
\enn
it follows that
\be\label{4.33+}
\|\g_t\bm H^p-\mathscr{T}[\g_T\bm E^p]\|_{Y(0,T;\G_R)}
&=& \|\g_t(\bm H^p-B\wid{\bm H}^p)\|_{Y(0,T;\G_R)}\\ \no
&\le& C\sup_{\bm v\in X(0,T;\Om^{{\rm PML}})}
\frac{|\int_0^T\langle\g_t(\bm H^p-B\wid{\bm H}^p),\g_T\bm v\rangle_{\G_R}dt|}
{\|\bm v\|_{X(0,T;\Om^{{\rm PML}})}}.
\en
For any $\bm v\in X(0,T;\Om^{{\rm PML}})$ we have $\g_t\bm v=0$ on $\G_{\rho}$,
and so, integrating by parts gives
\be\no
\int_0^T\langle\g_t(\bm H^p-B\wid{\bm H}^p),\g_T\bm v\rangle_{\G_R}dt
&=&\int_0^T\left(\nabla\times(\bm H^p-B\wid{\bm H}^p),\bm v\right)_{\Om^{{\rm PML}}}dt\\ \label{4.33+1}
&&-\int_0^T\left((\bm H^p-B\wid{\bm H}^p),\nabla\times\bm v\right)_{\Om^{{\rm PML}}}dt.\qquad\;\;
\en
Now, for $\bm v\in X(0,T;\Om^{{\rm PML}})$ it follows by noting the definition of $\bm v^*$ that
\be\label{4.33a}
&&\int_0^T\left((\bm H^p-B\wid{\bm H}^p),\nabla\times\bm v\right)_{\Om^{{\rm PML}}}dt\\ \no
&&\quad=\left((\bm H^p-B\wid{\bm H}^p),\nabla\times\bm v^*\right)_{\Om^{{\rm PML}}}\Big|_{t=T}
-\int_0^T\left(\pa_t(\bm H^p-B\wid{\bm H}^p),\nabla\times\bm v^*\right)_{\Om^{{\rm PML}}}dt.
\en
By the initial condition of $\bm H^p$ and $\wid{\bm H}^p$ we know that
$(\bm H^p-B\wid{\bm H}^p)\big|_{t=0}=0$, and thus
\ben
\left((\bm H^p-B\wid{\bm H}^p),\nabla\times\bm v^*\right)_{\Om^{{\rm PML}}}\Big|_{t=T}
=\left(\int_0^T\pa_t(\bm H^p-B\wid{\bm H}^p)dt,\nabla\times\bm v^*\big|_{t=T}\right)_{\Om^{{\rm PML}}}.
\enn
Combining this and (\ref{4.33a}) implies that
\be\no
&&\left|\int_0^T\left((\bm H^p-B\wid{\bm H}^p),\nabla\times\bm v\right)_{\Om^{{\rm PML}}}dt\right|\\
&&\quad\le 2\max_{0\le t\le T}\|\nabla\times\bm v^*\|_{L^2(\Om^{{\rm PML}})}
\int_0^T\left\|\pa_t(\bm H^p-B\wid{\bm H}^p)\right\|_{L^2(\Om^{{\rm PML}})}dt.\label{4.33+2}
\en
Using (\ref{4.33+}), (\ref{4.33+1}) and (\ref{4.33+2}) gives
\ben\label{4.42}
&&\|\g_t\bm H^p-\mathscr{T}[\g_T\bm E^p]\|_{Y(0,T;\G_R)}\\ \no
&&\quad\les\int_0^T\|\nabla\times(\bm H^p-B\wid{\bm H}^p)\|_{L^2({\Om^{{\rm PML}}})}dt
+\int_0^T\|\pa_t(\bm H^p-B\wid{\bm H}^p)\|_{L^2({\Om^{{\rm PML}}})}dt.
\enn
This together with (\ref{4.32}) leads to
\ben
&&\sup_{0\le t\le T}\left(\|\nabla\times\bm u^*\|_{L^2(\Om_R)^3}+\|\bm u\|_{L^2(\Om_R)^3}\right)\\
&&\qquad\les\int_0^T\|\nabla\times(\bm H^p-B\wid{\bm H}^p)\|_{L^2({\Om^{{\rm PML}}})}dt
+\int_0^T\|\pa_t(\bm H^p-B\wid{\bm H}^p)\|_{L^2({\Om^{{\rm PML}}})}dt.
\enn
Since $(\bm E^p-B\wid{\bm E}^p,\bm H^p-B\wid{\bm H}^p)$ satisfies the problem (\ref{3.12}) with
$\bm\xi=\g_t(B\wid{\bm E}^p|_{\G_{\rho}})$, it follows by (\ref{3.15b}) in Theorem \ref{thm3.3} that
\be\no
&&\sup_{0\le t\le T}\left(\|\nabla\times\bm u^*\|_{L^2(\Om_R)^3}+\|\bm u\|_{L^2(\Om_R)^3}\right)\\
&&\qquad\les(1+\sig_0T)^3T^{3/2}\|\g_t(B\wid{\bm E}^p)\|_{H^2(0,T;H^{-1/2}(\Dive,\G_\rho))}.\label{4.54}
\en
We now estimate the norm on the right-hand side of the inequality (\ref{4.54}).
By the boundedness of the trace operator $\g_t$ and the Parseval identity (\ref{A.5}) we have
\be\label{4.55}
\|\g_t(B\wid{\bm E}^p)\|^2_{H^2(0,T;H^{-1/2}(\Dive,\G_\rho))}
&\les&\|B\wid{\bm E}^p\|^2_{H^2(0,T;H(\curl,\;\Om^{\PML}))}\\ \no
&=&\int_0^T\left[\sum_{l=0}^2\|B\pa^l_t\wid{\bm E}^p\|^2_{H(\curl,\;\Om^{\PML})}\right]dt\\ \no
&\le& e^{2s_1T}\int_0^{\infty}e^{-2s_1t}
   \left[\sum_{l=0}^2\|B\pa^l_t\wid{\bm E}^p\|^2_{H(\curl,\;\Om^{\PML})}\right]dt\\ \no
&\les& e^{2s_1T}\left[1+\frac{\sigma_0}{s_1}\right]^4\int_{-\infty}^{\infty}
\left[\sum_{l=0}^2\|s^l\ch{\wid{\bm E}^p}\|^2_{H(\wid{\curl},\;\Om^{\PML})}\right]ds_2.
\en
By (\ref{4.33}), Theorem \ref{thm5.4} and the boundedness of $\g_T$ and $\g_t$ it is obtained that
\be\no
&&\sum_{l=0}^2\|s^l\ch{\wid{\bm E}^p}\|^2_{H(\wid{\curl},\;\Om^{\PML})}\\ \no
&&\qquad\quad\les s_1^{-4}d^4(1+s_1^{-1}\sigma_0)^{6}e^{-2\sqrt{\vep\mu}\rho\hat{\sigma}(\rho)}
\Big[(1+|s|^4)\sum_{l=0}^2\|s^l\g_t(\wid{\curl}\,\ch{\wid{\bm E}^p})\|^2_{H^{-1/2}(\Dive,\G_R)}\\ \no
&&\qquad\qquad+(1+|s|^6)\sum_{l=0}^2\|s^{l}\g_t(\ch{\bm E}^p)\|^2_{H^{-1/2}(\Dive,\G_R)}\Big]\\ \no
&&\qquad\quad\les s_1^{-4}d^4(1+s_1^{-1}\sigma_0)^{6}e^{-2\sqrt{\vep\mu}\rho\hat{\sigma}(\rho)}
\Big[(1+|s|^4)|s|^2(|s|^2+|s|^{-2})\\ \no
&&\qquad\qquad\sum_{l=0}^2\|s^l\g_T(\ch{\wid{\bm E}^p})\|^2_{H^{-1/2}(\Curl,\G_R)}+(1+|s|^6)
\sum_{l=0}^2\|s^{l}\g_t(\ch{\bm E}^p)\|^2_{H^{-1/2}(\Dive,\G_R)}\Big]\\ \no
&&\qquad\quad\les s_1^{-4}d^4(1+s_1^{-1}\sigma_0)^{6}e^{-2\sqrt{\vep\mu}\rho\hat{\sigma}(\rho)}
\sum_{l=0}^{6}\|s^l\ch{\bm E}^p\|^2_{H(\curl,\;\Om_R)}\\ \label{4.55a}
&&\qquad\quad\les s_1^{-6}d^4(1+s_1^{-1}\sigma_0)^{8}e^{-2\sqrt{\vep\mu}\rho\hat{\sigma}(\rho)}
  \sum_{l=0}^7\|s^{l}\ch{\bm J}\|^2_{L^2(\Om_R)^3},
\en
where we have used Lemma \ref{lem3.1} and the upper bound estimate (\ref{B-bound}) of
the EtM operator $\mathscr{B}$.
Combining (\ref{4.55}), (\ref{4.55a}) and the Parseval identity (\ref{A.5}), implies that
\be\no
&&\|\g_t(B\wid{\bm E}^p)\|^2_{H^2(0,T;H^{-1/2}(\Dive,\G_\rho))}\\ \no
&&\qquad\les e^{2s_1T}s_1^{-6}d^4(1+s_1^{-1}\sigma_0)^{12}e^{-2\sqrt{\vep\mu}\rho\hat{\sigma}(\rho)}
\sum_{l=0}^7\int_0^{\infty}e^{-2s_1t}\|\pa_t^{l}{\bm J}\|^2_{L^2(\Om_R)^3}dt\\ \label{4.55+}
&&\qquad\les e^{2s_1T}s_1^{-6}d^4(1+s_1^{-1}\sigma_0)^{12}e^{-2\sqrt{\vep\mu}\rho\hat{\sigma}(\rho)}
\|\bm J\|^2_{H^7(0,T;L^2(\Om_R)^3)},
\en
where we used the assumptions (\ref{assumption}) and (\ref{assumption1}) to get the last inequality.

Now, by (\ref{4.0}) we have
\ben
\rho\hat{\sig}(\rho)=\frac{\sig_0d}{m+1}.
\enn
It is obvious that $m$ should be chosen small enough to ensure the rapid convergence
(thus we need to take $m=1$). Since $s_1^{-1}=T$ in (\ref{4.55+}), and by using (\ref{4.54})
we obtain the required estimate (\ref{4.43}) on noting the definition (\ref{4.4b}) of $\bm u$
and $\bm u^*$ and the relation $\nabla\times\bm u^*=-\mu(\bm{H}-\bm H^p)$.
The proof is thus complete.
\end{proof}

\begin{remark}\label{re5} {\rm
Theorem \ref{thm5.5} implies that, for large $T$ the exponential convergence of the PML method
can be achieved by enlarging the thickness $d:=\rho-R$ or the PML absorbing parameter $\sigma_0$
which increases as $\ln T$.
}
\end{remark}

\section{Conclusions}\label{sec6}

In this paper, an effective PML method has been proposed in the three-dimensional spherical coordinates
for solving time-domain electromagnetic scattering problems,
based on the real coordinate stretching technique associated with $[\Rt(s)]^{-1}$ in the frequency domain.
The well-posedness and stability estimates of the truncated PML problem in the time domain
have been established by using the Laplace transform and energy method.
The exponential convergence of the PML method has also been proved in terms of the thickness and
absorbing parameters of the PML layer, based on the stability estimates of solutions
of the truncated PML problem and the exponential decay estimates of the stretched dyadic Green's function
for the Maxwell equations in the free space.

Our method can be extended to other electromagnetic scattering problems, such as scattering
by inhomogeneous media or bounded elastic bodies as well as scattering in a two-layered medium.
We hope to report such results in the future.

\section*{Acknowledgements}

This work was partly supported by the NNSF of China grants 91630309 and 11771349.
We thank the reviewers for their carefully reading the paper and for their constructive
and invaluable comments and suggestions, leading to improvements of the paper.

\end{document}